\documentclass[11pt]{article}
\usepackage{graphicx}
\usepackage{amsmath}
\usepackage{amssymb}
\usepackage{amsfonts}
\usepackage{enumerate}
\usepackage{latexsym}
\usepackage{epsfig}

\newtheorem{theorem}{Theorem}[section]
\newtheorem{lemma}[theorem]{Lemma}

\def\proofbox{\begin{picture}(6.5,6.5)
\put(0,0){\framebox(6.5,6.5){}}\end{picture}}
\newenvironment{proof}{\noindent{\it Proof.\quad}}{\hfill\proofbox}

\begin{document}

\begin{center} {\bf \LARGE Injective Simplicial Maps of the Arc Complex}\\

\vspace{0.09in}

{\bf \LARGE on Nonorientable Surfaces}

\vspace{0.15in}
\Large{Elmas Irmak}\\
\vspace{0.09in}
\end{center}



\begin{abstract} We prove that each injective simplicial map from the arc complex
of a compact, connected, nonorientable surface with nonempty boundary to itself is induced
by a homeomorphism of the surface. We also prove that the automorphism group of the arc complex
is isomorphic to the quotient of the mapping class group of the surface by its
center. \end{abstract}




\section{Introduction}

Let $N$ be a compact, connected, nonorientable surface of genus $g$ with $r \geq 1$ boundary components.
The genus of a nonorientable surface is the maximum number of projective planes in a connected sum decomposition.
Mapping class group, $Mod_N$, of $N$ is defined to be the group of isotopy classes of all self-homeomorphisms of $N$.
An arc $a$ on $N$ is called \textit{properly embedded} if
$\partial a \subseteq \partial N$ and $a$ is transversal to
$\partial N$. $a$ is called \textit{nontrivial} (or
\textit{essential}) if $a$ cannot be deformed into $\partial N$ in
such a way that the endpoints of $a$ stay in $\partial N$ during
the deformation. The \textit{complex of arcs}, $\mathcal{A}(N)$, on $N$ is an abstract simplicial complex.
Its vertices are the isotopy classes of nontrivial properly embedded arcs on $N$.
A set of vertices forms a simplex if these vertices can be represented by pairwise disjoint arcs.

The main results of this paper:

\begin{theorem} Let $N$ be a compact, connected, nonorientable surface of genus $g$ with
$r \geq 1$ boundary components. If $\lambda : \mathcal{A}(N) \rightarrow \mathcal{A}(N)$ is
an injective simplicial map then $\lambda$ is induced by a homeomorphism $h : N \rightarrow N$
(i.e $\lambda([a]) = [h(a)]$ for every vertex $[a]$ in $\mathcal{A}(N))$.\end{theorem}

\begin{theorem} Let $N$ be a compact, connected, nonorientable surface of genus $g$ with
$r \geq 1$ boundary components. Then $Aut(\mathcal{A}(N)) \cong Mod_N /Z(Mod_N)$.\end{theorem}

Both of these theorems were proven for orientable surfaces by Irmak-McCarthy in \cite{IrM}.
We follow the outline of Irmak-McCarthy \cite{IrM} paper. The main results in [IrM]: Let $R$ be a compact,
connected, orientable surface of genus $g$ with
$r \geq 1$ boundary components. If $\lambda : \mathcal{A}(R) \rightarrow \mathcal{A}(R)$ is
an injective simplicial map then $\lambda$ is induced by a homeomorphism $h : R \rightarrow R$, and
$Aut(\mathcal{A}(R)) \cong Mod_R /Z(Mod_R)$.

\protect\nopagebreak\noindent\rule{1.5in}{.01in}{\vspace{0.005in}}

{\small{This paper was written during the author's stay at Mathematical Sciences Research Institute (MSRI) in Berkeley in Fall 2007.
The author thanks MSRI for the conditions provided during this stay.}}

Ivanov proved that the automorphisms of the arc complex which are induced by automorphisms of the curve complex are induced by
homeomorphisms of the surface for orientable surfaces in \cite{Iv}. Using this, he proved that each automorphism of the complex of curves
is induced by a homeomorphism of the surface if the genus at least two, and the automorphism group of
the curve complex is isomorphic to the mapping class group quotient by its center. As an application he proved that
isomorphisms between any two finite index subgroups are geometric. Ivanov's results were proven by Korkmaz in \cite{K1}
and Luo in \cite{L} for lower genus cases.

After Ivanov's work, mapping class group was viewed as the automorphism group
of various geometric objects on orientable surfaces, and this was used to have information on the algebraic structure
of the mapping class groups. These objects include Schaller's complex (see \cite{Sc} by Schaller), the
complex of pants decompositions (see \cite{M} by Margalit), the complex of nonseparating curves (see \cite{Ir3} by Irmak), the complex of separating curves (see \cite{BM1} by Brendle-Margalit, and \cite{MV} by McCarthy-Vautaw), the complex of Torelli geometry (see \cite{FIv} by Farb-Ivanov), the Hatcher-Thurston complex (see \cite{IrK} by Irmak-Korkmaz), and complex of arcs (see \cite{IrM} by Irmak-McCarthy). As applications, Farb-Ivanov proved that the automorphism group of the Torelli subgroup is isomorphic to the mapping class group in \cite{FIv}, and McCarthy-Vautaw extended this result to $g \geq 3$ in \cite{MV}.

Some similar results on simplicial maps and the applications for orientable surfaces are as follows: Irmak proved that superinjective simplicial maps of the curve complex are induced by homeomorphisms of the surface to classify injective homomorphisms from finite
index subgroups of the mapping class group to the whole group (they are geometric except for closed genus two surface) for genus at least two in \cite{Ir1}, \cite{Ir2}, \cite{Ir3}. Behrstock-Margalit and Bell-Margalit proved these results for lower genus cases in \cite{BhM} and in \cite{BeM}. Brendle-Margalit proved that superinjective simplicial maps of separating curve complex are induced by homeomorphisms, to prove that an injection from a finite index subgroup of
$K$ to the Torelli group, where $K$ is the subgroup of mapping class group
generated by Dehn twists about separating curves, is induced by a
homeomorphism in \cite{BM1}, \cite{BM2}. Shackleton proved that
injective simplicial maps of the curve complex are induced by homeomorphisms in \cite{Sc} (he also considers maps between different surfaces), and he obtained strong local co-Hopfian results for mapping class groups. Bell-Margalit proved that superinjective simplicial maps of the curve complex are onto. For nonorientable odd genus surfaces Atalan-Ozan proved that the automorphism group of the curve complex is isomorphic to the mapping class group if $g + r \geq 6$ in \cite{A}.

\section{Mapping Class Groups and Complex of Arcs}

First we give definitions of some homeomorphisms on nonorientable surfaces and also a model of punctured $\mathbb{RP}^2$ as
given in \cite{K2}.

Elementary braid: If $N$ has at least two boundary components then an elementary braid is defined as follows: Let $a$ be an arc joining two distinct boundary components $\partial_i$ and $\partial_j$ of $N$. A regular neighborhood of $a \cup \partial_i \cup \partial_j$ is a pair of pants $P$. Let $z$ denote the boundary component of $P$ which is different from
$\partial_i$ and $\partial_j$. Let $\sigma_a$ be the isotopy class of a homeomorphism $N \rightarrow N$ supported in $P$ interchanging $\partial_i$ and $\partial_j$ such that $\sigma_a ^2 = t_z$ where $t_z$ is the Dehn twist about $z$ (the isotopy class of a twist homeomorphism about two-sided simple closed curve $z$). The class $\sigma_a$ will be called an {\it elementary braid}.

Boundary slide: Consider a Mobius band $M$ with one hole. By sliding the boundary component along the core of $M$ we get a homeomorphism of $M$ fixing a neighborhood of $\partial M$. If $M$ is embedded in a surface $N$, this homeomorphism can be extended to a homeomorphism of $N$ by the identity. This homeomorphism and its isotopy class will be called a boundary slide.

\begin{figure}
\begin{center}
\epsfxsize=2.3in \epsfbox{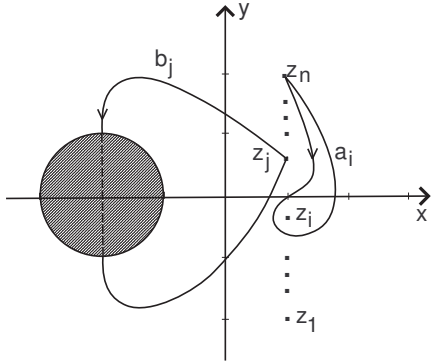}
\caption{Punctured $\mathbb{RP}^2$}
\label{PuncturedRP2}
\end{center}
\end{figure}

A model of punctured $\mathbb{RP}^2$: Consider the two sphere as the one point compactification of the xy-plane. Let $D= \{ (x, y) : (x+2)^2 + y^2 < 1\}$. For a positive integer $n$, take $n$ distinct points $z_i$ on the line segment $x=1, -2 \leq y \leq 2$, so that $z_1= (1, -2), z_n=(1, 2)$ and the second coordinates of $z_i$ are in ascending order. Remove the disc $D$ and identify the antipodal points on the boundary. This is a model for the surface $N$ of genus 1 with $n$ punctures as shown in Figure 1. Let $b_j$ be as shown in the figure. Let $v_j$ be the boundary slide (puncture slide) obtained by sliding $z_j$ once along $b_j$. Then $v_j^2$ is a Dehn twist about the boundary component of a regular neighborhood of $b_j$.

Now we will prove our main results for $(g, r) = (1,1)$, $(g, r)=(1,2)$, and $(g, r)=(2,1)$. We will then give a general argument for the proof of the remaining cases.

\begin{theorem}
\label{mainthm1} Let $N$ be a compact, connected, nonorientable surface of genus $g$ with
$r \geq 1$ boundary components. If $(g, r)=(1,1)$ or $(g, r)=(1,2)$ or $(g, r)=(2,1)$,
then $Aut(\mathcal{A}(N)) \cong Mod_N /Z(Mod_N)$.\end{theorem}

\begin{proof} Case (i): Suppose that $(g, r)=(1,1)$. $Mod_N$ is isomorphic to the mapping class group of a nonorientable surface $S$ of genus 1 with 1 puncture, which is isomorphic to $\mathbb{Z}_2$ and it is generated by the boundary slide $v_1$ as defined above by Theorem 4.1 in Korkmaz \cite{K2}. So, $Mod_N /Z(Mod_N)$ is trivial and also $Aut(\mathcal{A}(N))$ is trivial since $\mathcal{A}(N)$ has only one vertex in this case.

Case (ii): Suppose that $(g, r)=(1,2)$. $Mod_N$ is isomorphic to the mapping class group of a nonorientable surface $S$ of genus 1 with 2 punctures, which is isomorphic to the the Dihedral group of order 8 by Corollary 4.6 in Korkmaz \cite{K2}. It has center $\mathbb{Z}_2$ generated by $v_1 v_2$, where $v_i$ is the puncture slide as defined before. We have $Mod_N /Z(Mod_N) \cong \mathbb{Z}_2 \times \mathbb{Z}_2$. In this case $\mathcal{A}(N)$ is given by Scharleman in \cite{Sc}. He proves that the complex has 8 vertices, and it is as shown in Figure 2. It is easy to see that $Aut(\mathcal{A}(N))$ is $\mathbb{Z}_2 \times \mathbb{Z}_2$. Hence, we get $Aut(\mathcal{A}(N)) \cong Mod_N /Z(Mod_N) \cong \mathbb{Z}_2 \times \mathbb{Z}_2$.

Case (iii): Suppose that $(g, r)=(2, 1)$. $Mod_N$ is isomorphic to the mapping class group of a nonorientable surface $S$ of genus 2 with 1 puncture, and the structure of the mapping class group of $N$ is given by Stukow in \cite{St} as follows: Let $e$ and $b$ be as in the Figure 3 (i). Let $v$ be the puncture slide along $b$. Cut $N$ along $e$ to get a cylinder with one puncture. Reflection of this cylinder across the circle parallel to boundary components and passing through the puncture induces a homeomorphism $\sigma$ such that $\sigma(e) = e^{-1}$. Stukow proves that $Mod_N = (<t_e> \rtimes <v>) \times < \sigma > \cong (\mathbb{Z} \rtimes \mathbb{Z}_2) \times \mathbb{Z}_2$ and $Z(Mod_N) \cong \mathbb{Z}_2  = < \sigma > $ (Theorem A.5 and Corollary A.6 in \cite{St}). So, we get $Mod_N /Z(Mod_N) \cong \mathbb{Z} \rtimes \mathbb{Z}_2$.

In this case also $\mathcal{A}(N)$ is given by Scharleman in \cite{Sc}. Let $e, a, b, c$ be as in
Figure 3 (i). He proves that the vertices of $\mathcal{A}(N)$ are $\{[a], t_e ^n([b]), t_e ^n([c]): n \in \mathbb{Z} \}$, and the complex is as shown in Figure 3 (ii). It is easy to see that the automorphism group of this complex is $\mathbb{Z} \rtimes \mathbb{Z}_2$. So, we get $Aut(\mathcal{A}(N)) \cong Mod_N /Z(Mod_N) \cong \mathbb{Z} \rtimes \mathbb{Z}_2$.
\end{proof}

\begin{figure}
\begin{center}
\epsfxsize=2.3in \epsfbox{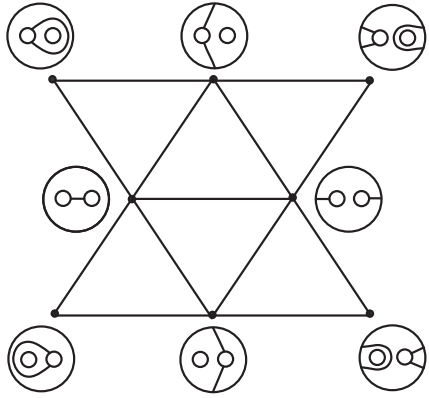}
\caption{Arc complex for $(g, r)=(1,2)$}
\label{arccomplexfor12}
\end{center}
\end{figure}

\begin{figure}
\begin{center}
\hspace{0.25cm} \epsfxsize=2.3in \epsfbox{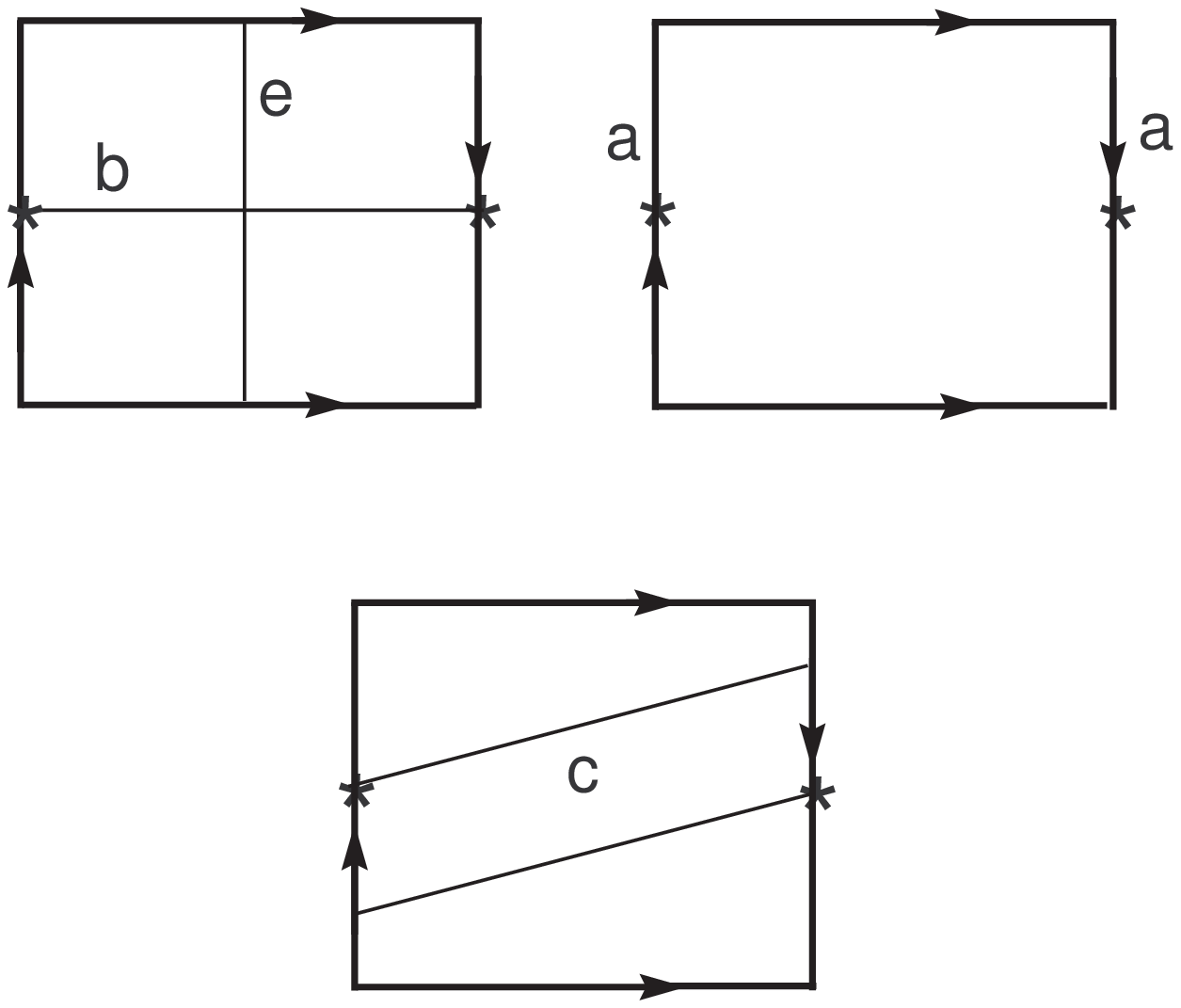}
\epsfxsize=3in \epsfbox{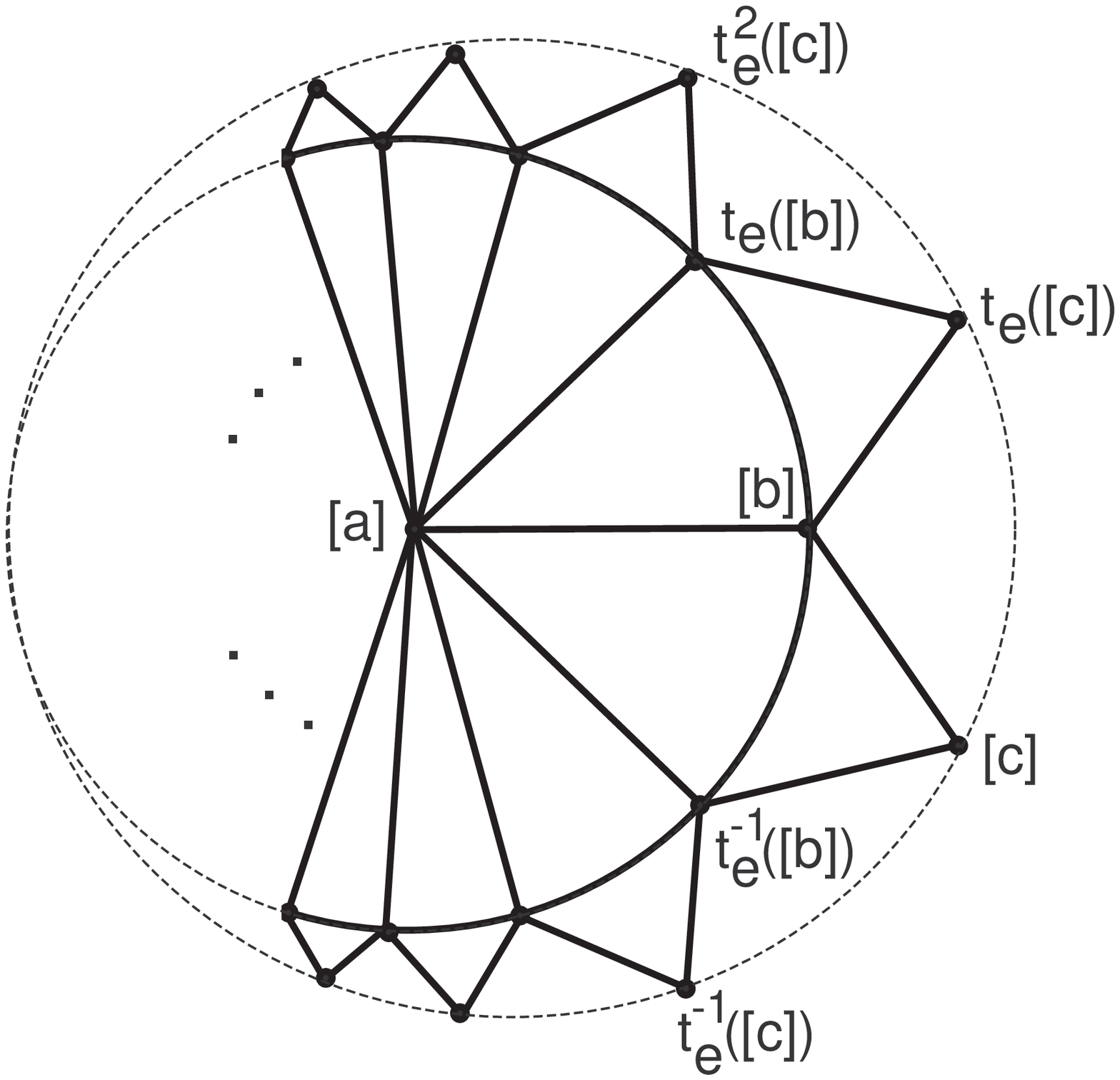}
\vspace{0.5cm}
\small{(i) \hspace{5cm} (ii) \hspace{1.1cm}}
\caption{Arc complex for $(g, r)=(2,1)$}
\label{arccomplexfor21}
\end{center}
\end{figure}

\begin{theorem}
\label{mainthm2} Let $N$ be a compact, connected, nonorientable surface of genus $g$ with
$r \geq 1$ boundary components. Suppose that $(g, r)=(1,1)$ or $(g, r)=(1,2)$ or $(g, r)=(2,1)$.
If $\lambda : \mathcal{A}(N) \rightarrow \mathcal{A}(N)$ is an injective simplicial map then $\lambda$
is induced by a homeomorphism $h : N \rightarrow N$.
\end{theorem}

\begin{proof}
For the cases $(g, r)=(1,1)$ and $(g, r)=(1,2)$, the proof follows from Theorem 2.1 as for both cases the arc complex has finitely many vertices so every injective simplicial map is an automorphism.

For $(g, r)=(2,1)$, we will prove that every injection is onto. Since $[a]$ is the only vertex which is connected to infinitely many vertices in the $\mathcal{A}(N)$ and $\lambda$ is injective $[a]$ is fixed by $\lambda$. Since the elements in $A= \{t_e ^n([b]): n \in \mathbb{Z} \}$ are the only vertices in $\mathcal{A}(N)$ which are connected to exactly five vertices by an edge and  $\lambda$ is injective, we have $\lambda(A) \subseteq A$. Similarly, since the elements in $B = \{t_e ^n([c]): n \in \mathbb{Z} \}$ are the only vertices in $\mathcal{A}(N)$ which are connected to exactly two vertices by an edge and $\lambda$ is injective, we have $\lambda(B) \subseteq B$. Let $w$ be a vertex of $\mathcal{A}(N)$ different from $[a]$.
Let $\Delta$ be a top dimensional simplex containing $[a]$ in $\mathcal{A}(N)$. Since $\lambda$ is injective and $\lambda([a])=[a]$, we see that $\lambda(\Delta)$ corresponds to a top dimensional simplex, $\Delta'$, containing $[a]$ in $\mathcal{A}(N)$. If $w$ is a vertex of $\Delta'$, then $w$ is in the image, so we are done. Suppose that $w$ is not a vertex of $\Delta'$. We will consider two cases:

Case (i): Suppose $w \in A$. Take a top dimensional simplex $\Delta''$ containing $w$ and $[a]$. It is easy to see that there exists a chain
$\Delta' = \Delta_0', \Delta_1', \cdots, \Delta_n' = \Delta''$ of top dimensional simplices containing $a$ in $\mathcal{A}(N)$, connecting $\Delta'$ to $\Delta''$ in such a way that the consecutive simplices
$\Delta_i', \Delta_{i+1}'$ have exactly one common face of codimension 1. Let $\Delta_1, \Delta_2$ be two distinct top dimensional simplices in $\mathcal{A}(N)$ different from $\Delta$, such that they both contain $[a]$ and also they have one common face of codimension 1 with $\Delta$. Since $\lambda$ is injective, top dimensional simplices are sent to top dimensional simplices by $\lambda$. Every edge containing $[a]$ is contained in exactly two top dimensional simplices $\mathcal{A}(N)$. Since $\lambda([a])=[a]$ and $\lambda(\Delta) = \Delta'$ either $\lambda(\Delta_1) = \Delta_1'$ or $\lambda(\Delta_2) = \Delta_1'$. This shows that every vertex of $\Delta_1'$ is in the image. By an inductive argument, using the above sequence we see that $w$ is in the image of $\lambda$.

Case (ii): Suppose $w \in B$. Take a top dimensional simplex $\Delta''$ containing $w$. Let $z, t \in A$ be the distinct vertices of $\Delta''$ different from $w$. By case (i) there exist vertices $x, y \in A$ such that $z=\lambda(x), t=\lambda(y)$. Consider the top dimensional simplex $\Delta'''$ containing $x$ and $y$ but not $[a]$. Let $k$ the vertex of $\Delta'''$ different from $x$ and $y$. Since $\lambda$ is injective and $\lambda([a])=[a]$, we see that $\lambda(\Delta''')$ is a top dimensional simplex containing $z$ and $t$ but not containing $[a]$. Since $\Delta''$ is the only top dimensional simplex containing $z$ and $t$ but not containing $[a]$, we see that $\Delta'' = \lambda(\Delta''')$ and $w=\lambda(k)$. Hence, $\lambda$ is onto. By using Theorem 2.1 we see that $\lambda$ is induced by a homeomorphism.
\end{proof}

\section{Triangulations}

In this section we assume that $(g, r) \neq (1,1)$,  $(g, r) \neq (1,2)$ and $(g, r) \neq (2,1)$. This is equivalent to having $g+r \geq 4$ since $r \geq 1$, $g \geq 1$.

Let $\lambda : \mathcal{A}(N) \rightarrow \mathcal{A}(N)$ be an injective simplicial map. We will prove some properties of $\lambda$. First we give some definitions.

Let $T$ be a set of pairwise disjoint nonisotopic arcs on $N$. $T$ is called a {\it triangulation of} $N$ if each component $\Delta$ of the surface $N_T$, obtained from $N$ by cutting $N$ along $T$, is a disc with boundary $\partial \Delta$ equal to a union of arcs, $a, b, c, d, e$, and $f$, where $a$, $b$, and $c$ correspond to elements of $T$ and $d$, $e$, and $f$ correspond
to arcs or circles in $\partial N$. $\Delta$ is called a {\it triangle of} $T$, and $a, b, c$ are called sides of $\Delta$. If $a$, $b$, and $c$
correspond to distinct elements of $T$, then $\Delta$ is called an {\it an embedded triangle of} $T$. Otherwise, $\Delta$ is called {\it a non-embedded triangle of} $T$. The phrase {\it triangle of} $T$ will also
be used to refer to the image of any component $\Delta$ of $N_T$, under the natural quotient map $q : N_T \rightarrow N$, together with the images of the
arcs $a$, $b$, and $c$ of $\partial \Delta$. Two distinct triangles of a triangulation $T$ are called {\it adjacent} w.r.t. $T$ if they have a common side. The images of $a, b$ and $c$ will also be called as sides of the triangle.

Let $T$ be a triangulation of $N$. Let $[T]$ be the set of isotopy classes of elements of $T$. Note that $[T]$ is a maximal simplex of $\mathcal{A}(N)$. Every maximal simplex $\sigma$ of $\mathcal{A}(N)$ is equal to $[T]$ for some triangulation $T$ of $N$.

Let $\lambda: \mathcal{A}(N) \rightarrow \mathcal{A}(N)$ be an injective simplicial map. Then $\lambda([T]) = [T']$ for some triangulation $T'$ of $N$ and $\lambda$ restricts to a correspondence $\lambda| : [T] \rightarrow [T']$
on the isotopy classes. Note that the triangulation $T'$ of $N$ is well defined up to isotopy on $N$. By using Euler
characteristic arguments we see that the number of arcs in a triangulation is $3g+3r-6$, and the number of triangles in a triangulation is $2g+2r-4$.

\begin{figure}
\begin{center}
\epsfxsize=1.3in \epsfbox{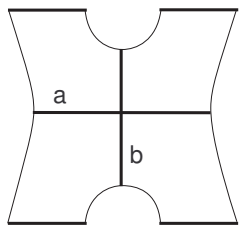}
\caption{Arcs intersecting once}
\label{intonefig}
\end{center}
\end{figure}

\begin{lemma}
\label{intone} If $\alpha$ and $\beta$ are two vertices in $\mathcal{A}(N)$ such that $i(\alpha, \beta)=1$ then $i(\lambda(\alpha), \lambda(\beta))=1$.
\end{lemma}

\begin{proof} Let $a$ and $b$ be representatives of $\alpha$ and $\beta$ respectively such that $a$ and $b$ intersect once as in Figure \ref{intonefig}. We complete $a$ to a triangulation $T_1$ of $N$ such that $(T_1 \setminus \{a\}) \cup \{b\}$ is also a triangulation of $N$. Let $T_2=(T_1 \setminus \{a\}) \cup \{b\}$. Let $\sigma_i$ be a simplex corresponding to triangulation $T_i$ for $i = 1, 2$. Let $\sigma_i'= \lambda(\sigma_i)$. Since $\lambda$ is injective, $\sigma_i'$ corresponds to triangulations on $N$ for $i=1,2$. Since $\sigma_1 \setminus \{[a]\} = \sigma_2 \setminus \{[b]\}$ and $\lambda$ is injective, we have $\sigma_1' \setminus \{\lambda([a])\} = \sigma_2' \setminus \{\lambda([b])\}$. Then it is easy to see that there exist  $a' \in \lambda(\alpha)$ and $b' \in \lambda(\beta)$ such that $a'$ and $b'$ intersect once.
\end{proof}

\begin{figure}
\begin{center}
\epsfxsize=3.7in \epsfbox{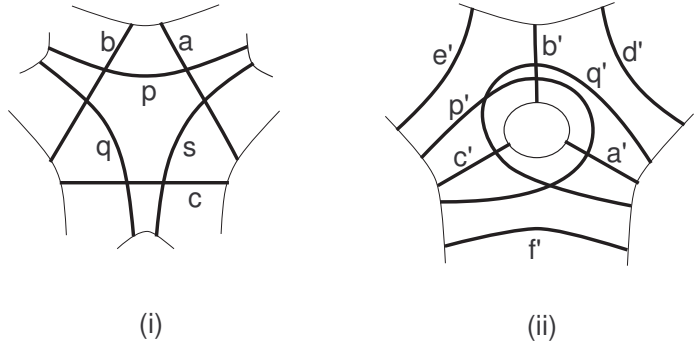}
\caption{Embedded triangle I}
\label{embedded 1}
\end{center}
\end{figure}

\begin{lemma}
\label{embedded} Let $a, b, c$ be properly embedded pairwise disjoint essential arcs on $N$ which form an embedded triangle on $N$. Then there exist $a' \in \lambda([a]), b' \in \lambda([b]), c' \in \lambda([c])$ such that $a', b', c'$ form an embedded triangle on $N$.
\end{lemma}

\begin{proof} Since $g + r \geq 4$ we can see that there exists a triangulation $T$ of $N$ containing $\{a, b, c\}$ such that the embedded triangle $\Delta$ formed by $a, b$ and $c$ on $N$ is adjacent to three distinct triangles $\Delta_a, \Delta_b$ and $\Delta_c$ of $T$ on $N$.

Note that $\partial \Delta_a$ is equal to union of arcs $a_1, x_1, a_2, x_2, a_3, x_3$ where $a_1, a_2, a_3$ correspond to elements of $T$ and each of $x_1, x_2, x_3$ correspond to an arc or a component of $\partial R$. W.L.O.G. we assume that $a_1$ corresponds to $a$ and $x_2$ is disjoint from $a_1$. Similarly, $\partial \Delta_b$ is equal to union of arcs $b_1, y_1, b_2, y_2, b_3, y_3$ where $b_1, b_2, b_3$ correspond to elements of $T$ and each of $y_1, y_2, y_3$ correspond to an arc or a component of $\partial R$. W.L.O.G. we assume that $b_1$ corresponds to $b$ and $y_2$ is disjoint from $b_1$. Likewise, $\partial \Delta_c$ is equal to union of arcs $c_1, z_1, c_2, z_2, c_3, z_3$ where $c_1, c_2, c_3$ correspond to elements of $T$ and each of $z_1, z_2, z_3$ correspond to an arc or a component of $\partial R$. W.L.O.G. we assume that $c_1$ corresponds to $c$ and $z_2$ is disjoint from $c_1$.

Let $p$ be a properly embedded essential arc connecting $x_2$ to $y_2$ and intersecting only $a$ and $b$ once and disjoint from each
element of $T \setminus \{a, b\}$ as shown in Figure \ref{embedded 1} (i). Let $q$ be a properly embedded essential arc connecting $y_2$ to $z_2$ and intersecting only $b$ and $c$ once and disjoint from $p$ and each element of $T \setminus \{b, c\}$ as shown in Figure \ref{embedded 1} (i). Let $s$ be a properly embedded essential arc connecting $x_2$ to $z_2$ and intersecting only $a$ and $c$ once and disjoint from $p \cup q$ and each element of $T \setminus \{a, c\}$ as shown in Figure \ref{embedded 1} (i). The arcs $p, q$ and $s$ form an embedded triangle on $N$. It is easy to see that $T_1 = (T \setminus \{a, b, c\}) \cup \{p, q, s\}$ is a triangulation on $N$.

Let $T'$ be a triangulation on $N$ such that $\lambda([T]) = [T']$. Let $a', b', c'$ be representatives of $\lambda([a]), \lambda([b]), \lambda([c])$ respectively, such that they are also in $T'$.

We have $i([p], [w]) = 0$ for every element $w \in T \setminus \{a, b\}$, $i([p], [a])=1$ and $i([p], [b])=1$. Since $\lambda$ is injective, by using Lemma \ref{intone} we see that $i(\lambda([p]), \lambda([w])) = 0$ for every element $w \in T \setminus \{a, b\}$, $i(\lambda([p]), \lambda([a]))= 1$ and $i(\lambda([p]), \lambda([b]))=1$. Hence, there exists $p' \in \lambda([p])$ such that $p'$ intersects $a'$ and $b'$ essentially once and does not intersect any other element of $T'$. Similarly, there exists $q' \in \lambda([q])$ such that $q'$ intersects $b'$ and $c'$ essentially once and does not intersect any other element of $T'$, and there exists $s' \in \lambda([s])$ such that $s'$ intersects $a'$ and $c'$ essentially once and does not intersect any other element of $T'$.

Since $p$, $q$, and $s$ are disjoint, $i([p],[q]) = i([q],[s]) = i([s],[p]) = 0$. Since $\lambda$ is a simplicial map, it follows that $i(\lambda([p]),\lambda([q])) = i(\lambda([q]),\lambda([s])) = i(\lambda([s]),\lambda([p])) = 0$. Hence, we may assume that $p'$, $q'$, and $s'$ are pairwise disjoint arcs on $N$.

Since the essential arc $p'$ on $N$ intersects $a'$ and $b'$ essentially once and is disjoint from the other elements of the triangulation $T'$ of $N$, there exists a triangle $\Delta_1$ of $T'$ on $N$ having sides corresponding to $a'$ and $b'$. Similarly, there exists a triangle $\Delta_2$ of $T'$ on $N$ having sides corresponding to $b'$ and $c'$, and a triangle $\Delta_3$ of $T'$ on $N$ having sides corresponding to $c'$ and $a'$. Let $d'$ be the side of $\Delta_1$ different from $a', b'$ in $T'$, let $e'$ be the side of $\Delta_2$ different from $b', c'$ in $T'$. Let $f'$ be the side of $\Delta_3$ different from $a', c'$ in $T'$.

If $d' = c'$ or $e' = a'$ or $f' = b'$, then $a', b', c'$ form an embedded triangle, so we are done. Assume that $d' \neq c'$, $e' \neq a'$, and $f' \neq b'$. Since $a'$, $b'$, and $c'$ are distinct arcs on $N$, we have $\Delta_1 \neq \Delta_2$, $\Delta_2 \neq \Delta_3$ and $\Delta_3 \neq \Delta_1$. Hence, $\Delta_1$, $\Delta_2$, and $\Delta_3$ are three distinct components of $N_{T'}$.

We need to consider cases depending on the placement
of the arcs $c'$, $e'$, and $f'$ on $\partial \Delta_2$ and $\partial \Delta_3$ and the possible gluing of the triangles $\Delta_1, \Delta_2, \Delta_3$ on $N$. These cases are shown in Figures \ref{embedded 1} (ii), \ref{embedded 2}, \ref{embedded 3} and \ref{embedded 4}.

\begin{figure}
\begin{center}
\epsfxsize=3.5in \epsfbox{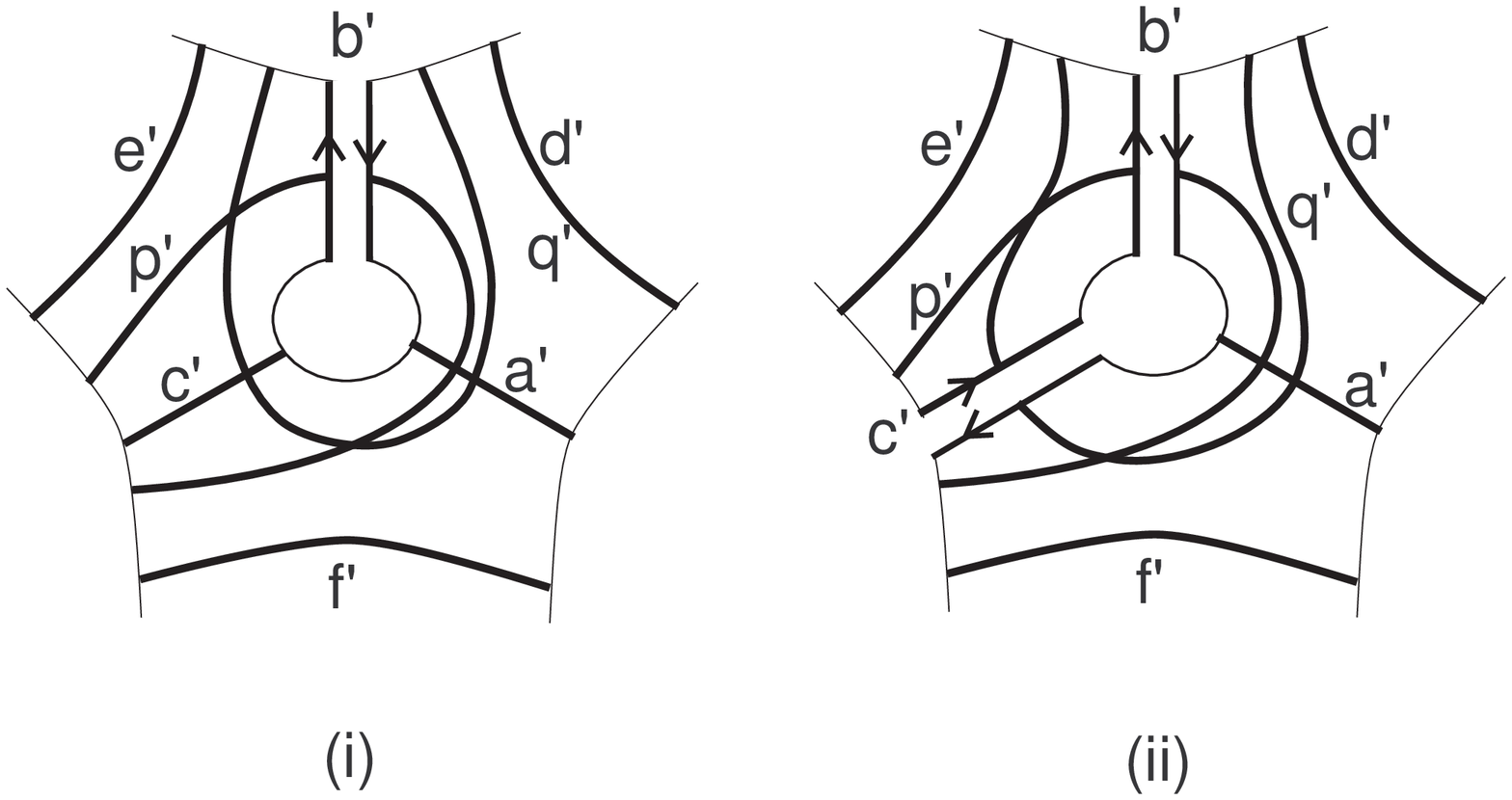} \hspace{0.1cm}
\epsfxsize=1.67in \epsfbox{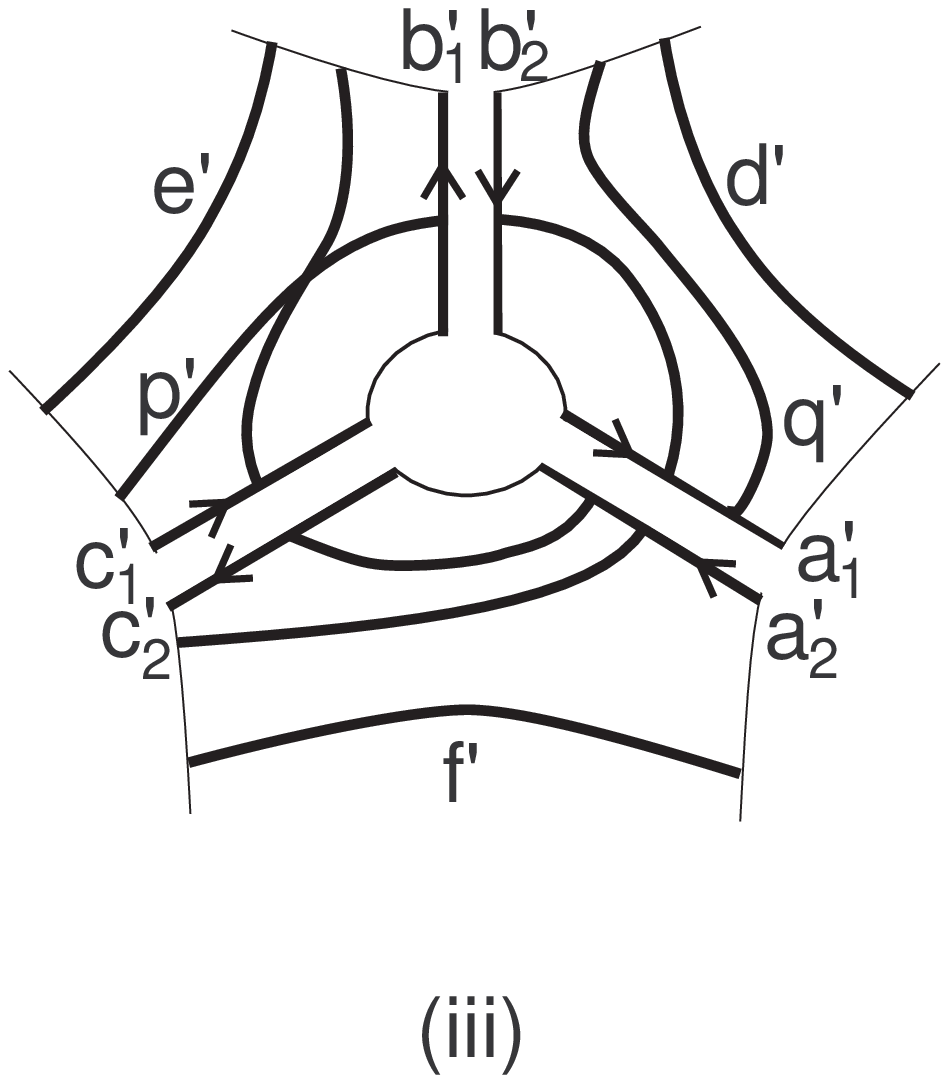}
\caption{Embedded triangle II}
\label{embedded 2}
\end{center}
\end{figure}

\begin{figure}
\begin{center}
\epsfxsize=3.7in \epsfbox{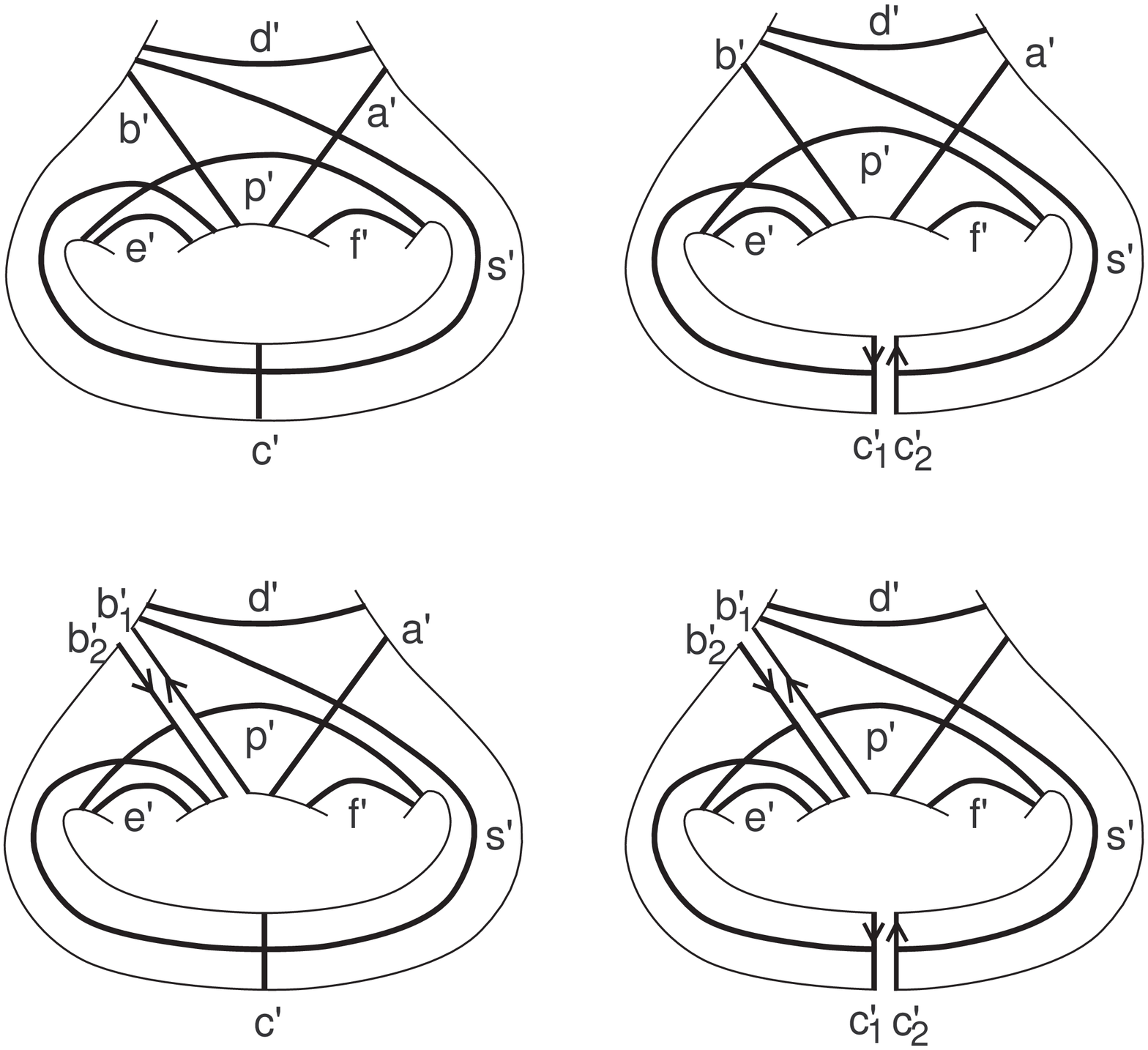}

\vspace{0.5cm}

\epsfxsize=3.7in \epsfbox{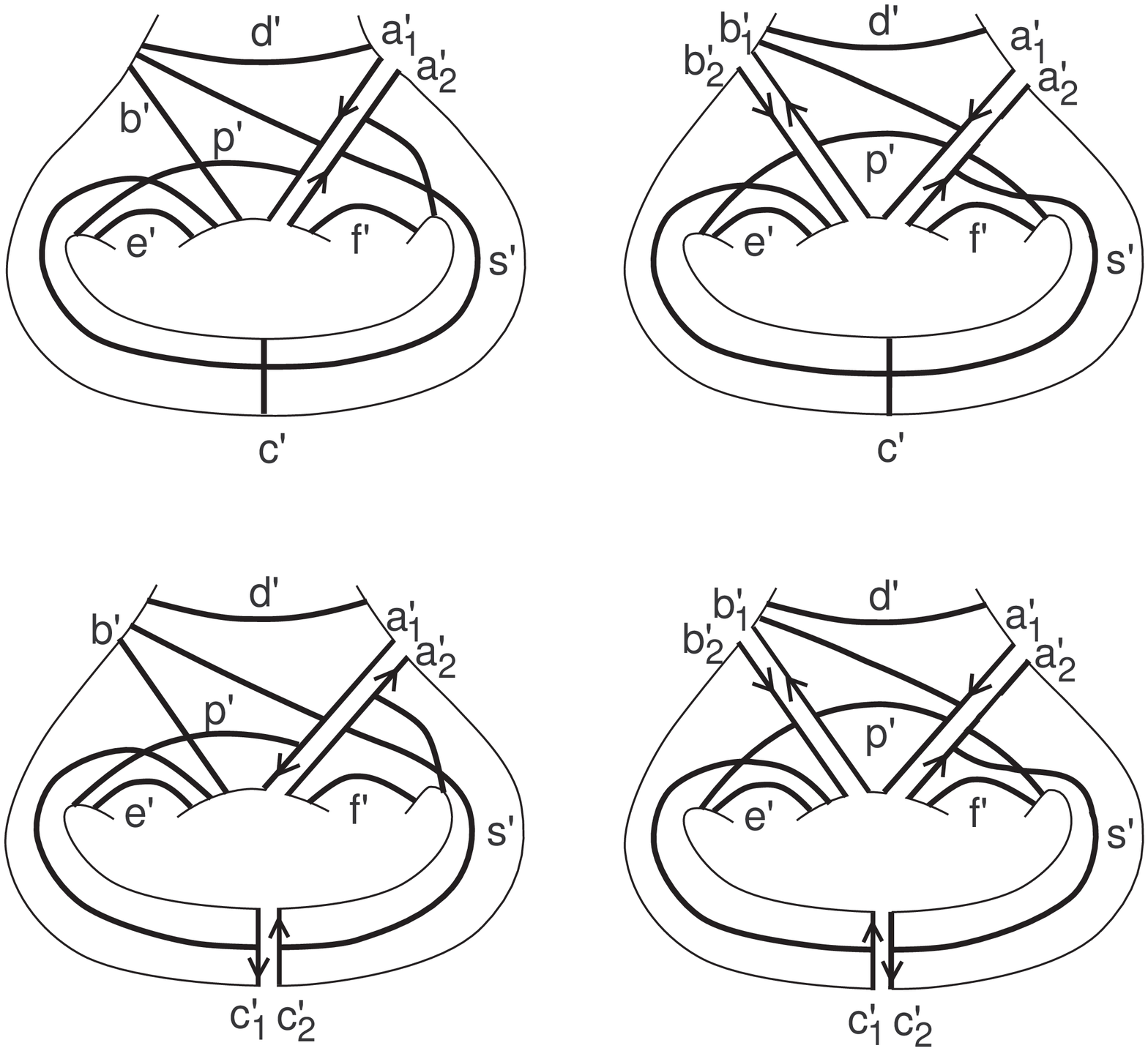}
\caption{Embedded triangle III}
\label{embedded 3}
\end{center}
\end{figure}

Case (i): Assume $a', b', c', d', e', f'$ are as shown in Figure \ref{embedded 1} (ii). Note that the arc $p'$ on $N$ representing $\lambda([p])$ intersects $b'$ and $a'$ once essentially and is disjoint from $e'$, $d'$, $f'$, and $c'$; and the arc $q'$ on $N$ representing $\lambda([q])$ intersects $b'$ and $c'$ once essentially and is disjoint from $e'$, $d'$, $f'$, and $a'$. But then we see that $p'$ and $q'$ intersect essentially, which gives a contradiction since $i(\lambda([p]), \lambda([q])) = 0$. Similarly, we can show that if $a', b', c', d', e', f'$ are as shown in Figure \ref{embedded 2} (i) or (ii) or (iii) we get a contradiction (see
Figure \ref{embedded 2}).

Case (ii): Assume $a', b', c', d', e', f'$ are as shown in Figure \ref{embedded 3} (i). As before, it follows that the arc $p'$ on $N$ representing $\lambda([p])$ intersects $b'$ and $a'$
once essentially and is disjoint from $e'$, $d'$, $f'$, and $c'$; and the arc $s'$ on $N$ representing $\lambda([s])$ intersects $a'$ and $c'$
once essentially and is disjoint from $e'$, $d'$, $f'$, $b'$. But then we see that
$p'$ and $s'$ intersect essentially
which gives a contradiction, since $i(\lambda([p]),\lambda([s])) = 0$. Similarly, we can show that if $a', b', c', d', e', f'$ are as shown in the remaining parts of Figure \ref{embedded 3} we get a contradiction.

\begin{figure}
\begin{center}
\epsfxsize=3.5in \epsfbox{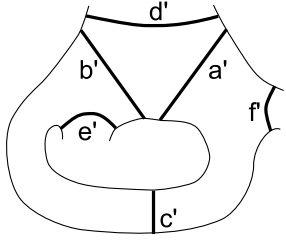}
\caption{Embedded triangle IV}
\label{embedded 4}
\end{center}
\end{figure}

Case (iii): If $a', b', c', d', e', f'$ are as shown in Figure \ref{embedded 4}, we get a contradiction as in the proof Case (ii). In the remaining cases, where the placement of the arcs on the triangles are as in Figure \ref{embedded 4}, we can see all possible gluing of these triangles formed by $e', b', c'$ and $a', c', f'$ and $a', b', d'$ as we saw in Figure \ref{embedded 3}, and get contradictions similarly as in Case (ii).

Hence, we see that either $d' = c'$ or $e' = a'$ or $f' = b'$ and, hence, we are done. \end{proof}\\

A non-embedded triangle is called a {\it regular non-embedded triangle} if it is homeomorphic to an annulus on the surface. In Figure \ref{regularfig} (i) we see a regular non-embedded triangle formed by $a$ and $b$. A non-embedded triangle is called a {\it twisted non-embedded triangle} if it is homeomorphic to a Mobius band on the surface. In Figure \ref{regularfig} (ii) we see a twisted non-embedded triangle formed by $a$ and $b$.

\begin{figure}
\begin{center}
\epsfxsize=1.2in \epsfbox{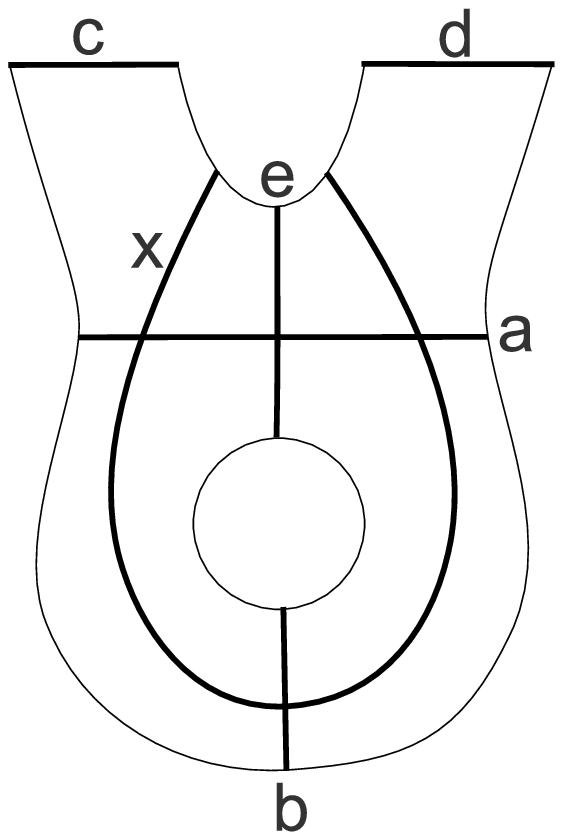}  \hspace{0.7cm} \epsfxsize=1.2in \epsfbox{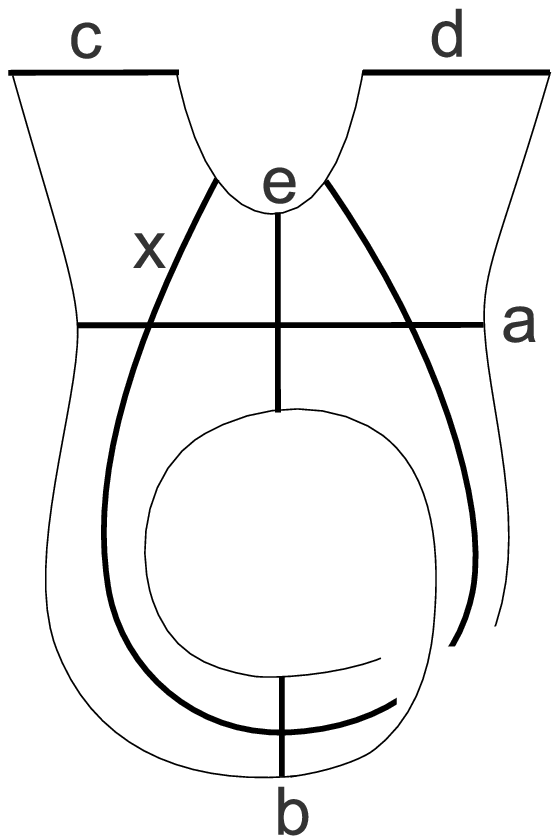}
\caption{(i) Arc configurations I \hspace{0.2cm} (ii) Arc configurations II}
\label{regularfig}
\end{center}
\end{figure}

\begin{lemma}
\label{regular} Let $a$ and $b$ be two properly embedded disjoint essential arcs on $N$ such that $a$ connects one boundary component to itself, and $a$ and $b$ form a regular non-embedded triangle on $N$. There exist $a' \in \lambda([a])$ and $b' \in \lambda([b])$ such that $a'$ connects one boundary component to itself, and $a'$ and $b'$ form a regular non-embedded triangle on $N$.\end{lemma}

\begin{proof} We complete $a$ and $b$ to an arc configuration consisting of arcs $\{a, b, c, d, e, x\}$ as shown in Figure \ref{regularfig} (i). By using Lemma \ref{embedded} we can see that since $b, c, d, e$ are pairwise disjoint essential arcs and $b, c, e$ and $b, d, e$ form embedded triangles on $N$ there exist $b' \in \lambda([b])$, $c' \in \lambda([c])$, $d' \in \lambda([d])$, $e' \in \lambda([e])$ such that $b', c', d', e'$ are pairwise disjoint and $b', c', e'$ and $b', d', e'$ form embedded triangles on $N$. Since $a$ intersects $e$ once and $a$ is disjoint from each of $b, c, d$, by using Lemma \ref{intone} we can choose $a' \in \lambda([a])$ such that $a'$ intersects $e'$ once and $a'$ is disjoint from each of $b', c', d'$. Since $a, c, d$ form an embedded triangle $a', c', d'$ form an embedded triangle.

Now there are four cases to consider as shown in Figure \ref{twisted?}:

Case (i): In the first case  $a', b'$ form a regular non-embedded triangle on $N$.

Case (ii): Suppose the gluing of the triangles formed by $b', c', e'$ and $b', d', e'$ are as shown in \ref{twisted?} (ii). Then $a', c', d'$ do not form an embedded triangle as shown in the figure which gives a contradiction.

Case (iii): Suppose the gluing of the triangles formed by $b', c', d'$ and $a', b', c'$ are as shown in \ref{twisted?} (iii). We consider the arc $x$ intersecting $b$ once and being disjoint from each of $c, d, e$. By using Lemma \ref{intone} we can choose $x' \in \lambda([x])$ such that $x'$ intersects $b'$ once and $x'$ is disjoint from each of $c', d', e'$. By using Lemma \ref{embedded} we get a contradiction as $c, x, d$ form an embedded triangle but $c', x', d'$ do not form an embedded triangle on $N$ as shown in Figure \ref{twisted?} (iii).

Case (iv): We get a contradiction as in Case (iii).

\end{proof}

\begin{figure}
\begin{center}
\epsfxsize=2.5in \epsfbox{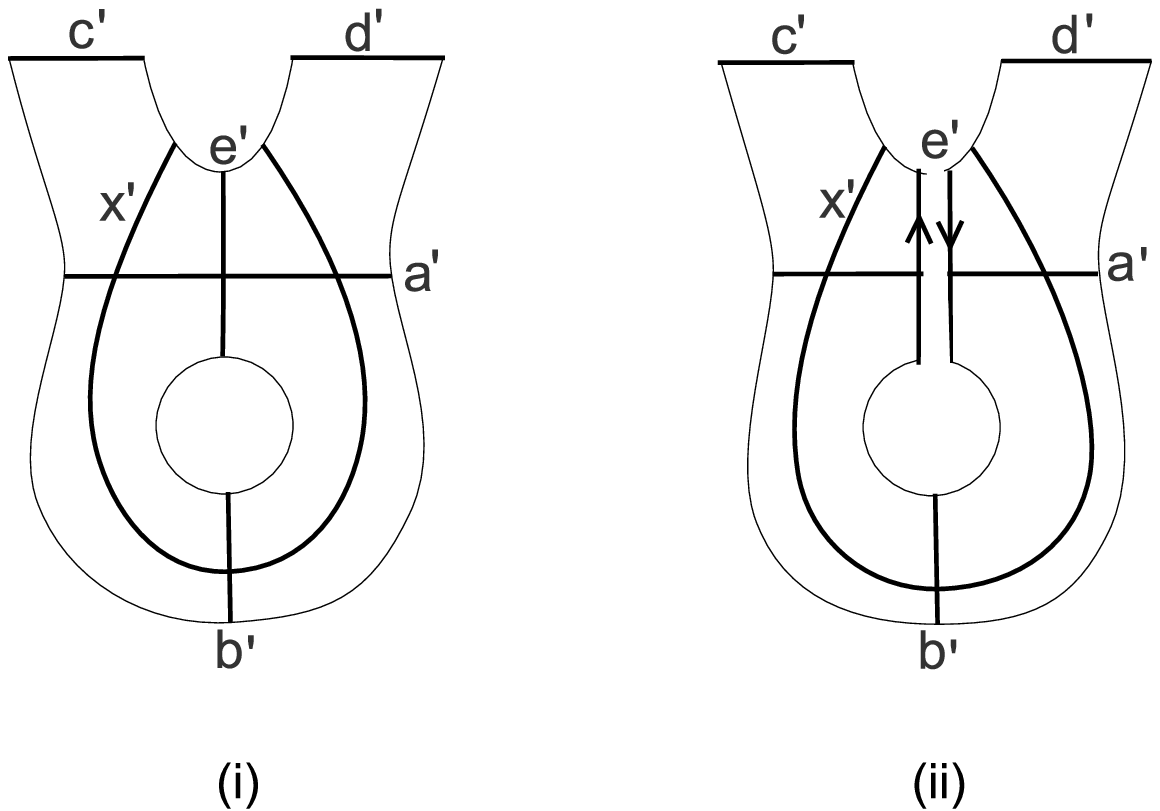}
\hspace{0.8cm}
\epsfxsize=2.35in \epsfbox{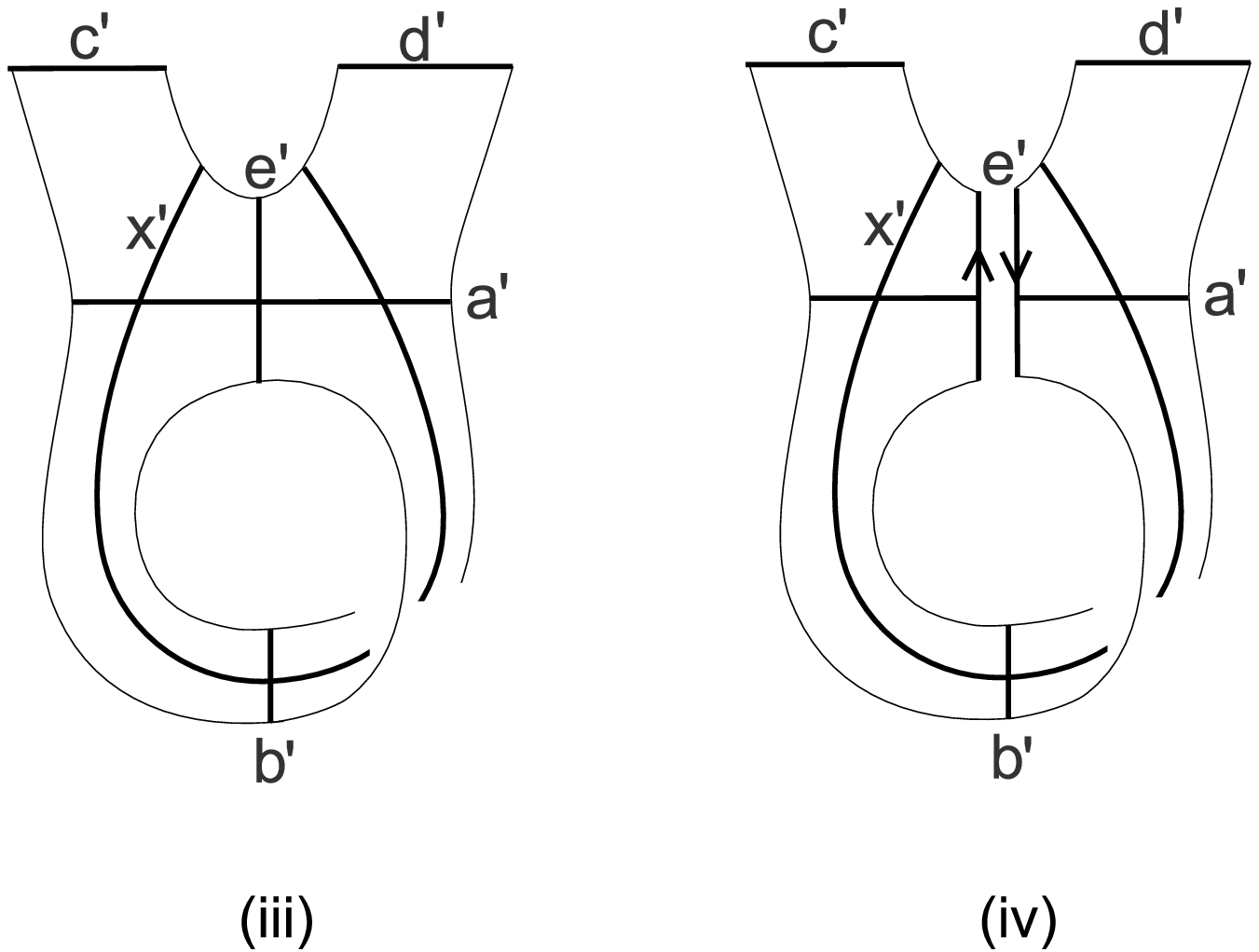}
\caption{Arc configurations III}
\label{twisted?}
\end{center}
\end{figure}

\begin{lemma}
\label{twisted}
Let $a$ and $b$ be two properly embedded disjoint essential arcs on $N$ such that $a$ connects one boundary component to itself, and $a$ and $b$ form a twisted non-embedded triangle on $N$. There exist $a' \in \lambda([a])$ and $b' \in \lambda([b])$ such that $a'$ connects one boundary component to itself and $a'$ and $b'$ form a twisted non-embedded triangle on $N$.
\end{lemma}

\begin{proof} We complete $a$ and $b$ to an arc configuration consisting of arcs $\{a, b, c, d, e, x\}$ as shown in Figure \ref{regularfig} (ii). By using Lemma \ref{embedded} we can see that since $b, c, d, e$ are pairwise disjoint essential arcs and $b, c, e$ and $b, d, e$ form embedded triangles on $N$ there exist $b' \in \lambda([b])$, $c' \in \lambda([c])$, $d' \in \lambda([d])$, $e' \in \lambda([e])$ such that $b', c', d', e'$ are pairwise disjoint and $b', c', e'$ and $b', d', e'$ form embedded triangles on $N$. Since $a$ intersects $e$ once and $a$ is disjoint from each of $b, c, d$, by using Lemma \ref{intone} we can choose $a' \in \lambda([a])$ such that $a'$ intersects $e'$ once and $a'$ is disjoint from each of $b', c', d'$. Since $a, c, d$ form an embedded triangle $a', c', d'$ form an embedded triangle.

Now there are four cases to consider as shown in Figure \ref{twisted?}:

Case (i): Suppose the gluing of the triangles formed by $b', c', e'$ and $b', d', e'$ are as shown in Figure \ref{twisted?} (i). We consider the arc $x$ intersecting $b$ once and being disjoint from each of $c, d, e$. By using Lemma \ref{intone} we can choose $x' \in \lambda([x])$ such that $x'$ intersects $b'$ once and $x'$ is disjoint from each of $c', d', e'$. By using Lemma \ref{embedded} we get a contradiction as $c, x, e$ form an embedded triangle but $c', x', e'$ do not form an embedded triangle on $N$ as shown in Figure \ref{regularfig} (ii) and Figure \ref{twisted?} (i).

Case (ii): Suppose the gluing of the triangles formed by $b', c', d'$ and $a', b', c'$ are as shown in \ref{twisted?} (ii). We get a contradiction as in Case (i).

Case (iii) In this case  $a', b'$ form a twisted non-embedded triangle on $N$.

Case (iv): Suppose the gluing of the triangles formed by $b', c', e'$ and $b', d', e'$ are as shown in \ref{twisted?} (iv). We see that $a', c', d'$ do not form an embedded triangle as shown in the figure which gives a contradiction.
\end{proof}

\begin{lemma}
\label{9} Let $a, b, c, d, e$ be essential pairwise disjoint nonisotopic properly embedded arcs on $N$. Suppose that there exists a subsurface $K$ of $N$ and a homeomorphism $\phi: (K, a, b, c, d, e) \rightarrow (K_0, a_0, b_0, c_0, d_0, e_0)$ where $K_0$ and $a_0, b_0, c_0, d_0, e_0$ are as shown in Figure \ref{yeni}. There exist $a' \in \lambda([a]), b' \in \lambda([b]), c' \in \lambda([c]), d' \in \lambda([d]), e' \in \lambda([e])$, $K' \subseteq N$ and a homeomorphism $\chi: (K', a', b', c', d', e') \rightarrow (K_0, a_0, b_0, c_0, d_0, e_0)$.
\end{lemma}

\begin{figure}
\begin{center}
\epsfxsize=1.2in \epsfbox{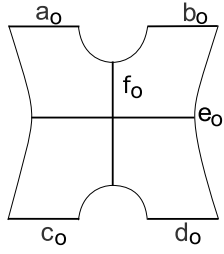}
\caption{Arc configurations IV}
\label{yeni}
\end{center}
\end{figure}

\begin{proof} Let $f_0$ be as shown in Figure \ref{yeni}. Let $f = \phi^{-1}(f_0)$. We see that $f$ is an essential properly embedded arc on $N$ such that $f$ intersects $e$ once and $f$ is disjoint from $a, b, c, d$. Since $a, b, c, d, e$ are disjoint, $a, b, e$ and $c, d, e$ form embedded triangles, by using that $\lambda$ is injective and the results of Lemma \ref{embedded}, we can choose
$a' \in \lambda([a]), b' \in \lambda([b]), c' \in \lambda([c]), d' \in \lambda([d]), e' \in \lambda([e])$ such that $a', b', c', d', e'$ are disjoint, and $a', b', e'$ and $c', d', e'$ form embedded triangles. Since $e, f$ intersect once and $f$ is disjoint from each of $a, b, c, d$, we can choose $f' \in \lambda([b])$ such that $e', f'$ intersect once and $f'$ is disjoint from each of $a', b', c', d'$. Since $a, c, f$ form an embedded triangle, $a', c', f'$ form an embedded triangle. Then the result of the lemma follows.
\end{proof}

\begin{lemma}
\label{9i} Let $a, b, c, d$ be essential pairwise disjoint nonisotopic properly embedded arcs on $N$. Suppose that there exists a subsurface $K$ of $N$ and a homeomorphism $\phi: (K, a, b, c, d) \rightarrow (K_0, a_0, b_0, c_0, d_0)$ where $K_0$ and $a_0, b_0, c_0, d_0$ are as shown in Figure \ref{9fig} (i). There exist $a' \in \lambda([a]), b' \in \lambda([b]), c' \in \lambda([c]), d' \in \lambda([d])$, $K' \subseteq N$ and a homeomorphism $\chi: (K', a', b', c', d') \rightarrow (K_0, a_0, b_0, c_0, d_0)$.
\end{lemma}

\begin{figure}
\begin{center}
\epsfxsize=2.7in \epsfbox{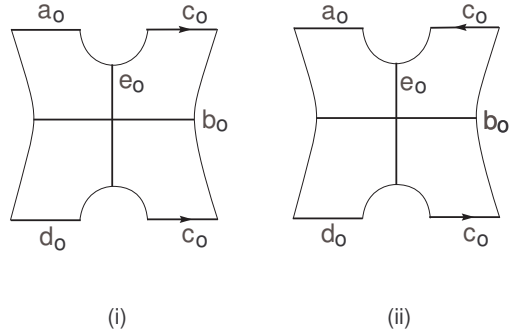}
\caption{Arc configurations V}
\label{9fig}
\end{center}
\end{figure}

\begin{proof} Let $e_0$ be as shown in Figure \ref{9fig} (i). Let $e = \phi^{-1}(e_0)$. We see that $e$ is an essential properly embedded arc on $N$ such that $e$ intersects $b$ once and $e$ is disjoint from $a, c, d$. Since $a, c, d, e$ are disjoint, $e, c$ form a regular non-embedded triangle, $a, d, e$ form an embedded triangle, by using that $\lambda$ is injective and the results of Lemma \ref{embedded} and Lemma \ref{regular}, we can choose
$a' \in \lambda([a]), e' \in \lambda([e]), c' \in \lambda([c]), d' \in \lambda([d])$ such that $a', c', d', e'$ are disjoint, $e', c'$ form a regular non-embedded triangle and $a', d', e'$ form an embedded triangle. Since $e, b$ intersect once and $e$ is disjoint from each of $a, c, d$, we can choose $b' \in \lambda([b])$ such that $e', b'$ intersect once and $e'$ is disjoint from each of $a', c', d'$, then the result of the lemma follows.
\end{proof}

\begin{lemma}
\label{9ii} Let $a, b, c, d$ be essential pairwise disjoint nonisotopic properly embedded arcs on $N$. Suppose that there exists a subsurface $K$ of $N$ and a homeomorphism $\phi: (K, a, b, c, d) \rightarrow (K_0, a_0, b_0, c_0, d_0)$ where $K_0$ and $a_0, b_0, c_0, d_0$ are as shown in Figure \ref{9fig} (ii). Then there exist $a' \in \lambda([a]), b' \in \lambda([b]), c' \in \lambda([c]), d' \in \lambda([d])$, $K' \subseteq N$ and a homeomorphism $\chi: (K', a', b', c', d') \rightarrow (K_0, a_0, b_0, c_0, d_0)$.
\end{lemma}

\begin{proof} Let $e_0$ be as shown in Figure \ref{9fig} (ii). Let $e = \phi^{-1}(e_0)$. We see that $e$ is an essential properly embedded arc on $N$ such that $e$ intersects $b$ once and $e$ is disjoint from $a, c, d$. Since $a, c, d, e$ are disjoint, $e, c$ form a twisted non-embedded triangle, $a, d, e$ form an embedded triangle, by using that $\lambda$ is injective and the results of Lemma \ref{embedded} and Lemma \ref{twisted}, we can choose
$a' \in \lambda([a]), e' \in \lambda([e]), c' \in \lambda([c]), d' \in \lambda([d])$ such that $a', c', d', e'$ are disjoint, $e', c'$ form a twisted non-embedded triangle and $a', d', e'$ form an embedded triangle. Since $e, b$ intersect once and $e$ is disjoint from each of $a, c, d$, we can choose $b' \in \lambda([b])$ such that $e', b'$ intersect once and $e'$ is disjoint from each of $a', c', d'$, then the result of the lemma follows.
\end{proof}

\begin{figure}
\begin{center}
\epsfxsize=2.7in \epsfbox{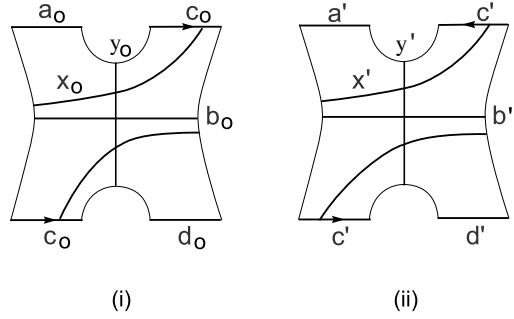}
\caption{Arc configurations VI}
\label{10fig}
\end{center}
\end{figure}

\begin{lemma}
\label{10i} Let $a, b, c, d$ be essential pairwise disjoint nonisotopic properly embedded arcs on $N$. Suppose that there exists a subsurface $K$ of $N$ and a homeomorphism $\phi:
(K, a, b, c, d) \rightarrow (K_0, a_0, b_0, c_0, d_0)$ where $K_0$ and $a_0, b_0, c_0, d_0$ are as shown in \ref{10fig} (i). There exist $a' \in \lambda([a]), b' \in \lambda([b]), c' \in \lambda([c]), d' \in \lambda([d])$, $K' \subseteq N$ and a homeomorphism $\chi: (K', a', b', c', d') \rightarrow (K_0, a_0, b_0, c_0, d_0)$.
\end{lemma}

\begin{proof} Let $x_0, y_0$ be as shown in Figure \ref{10fig} (i). Let $x = \phi^{-1}(x_0), y = \phi^{-1}(y_0)$. Since $a, b, c, d$ are pairwise disjoint, $a, b, c$ form an embedded triangle, and $b, c, d$ form an embedded triangle, by using that $\lambda$ is injective and the results of Lemma \ref{embedded} we can choose
$a' \in \lambda([a]), b' \in \lambda([b]), c' \in \lambda([c]), d' \in \lambda([d])$ such that $a', b', c', d'$ are pairwise disjoint, $a', b', c'$ form an embedded triangle and $b', c', d'$ form an embedded triangle. Since $y$ intersects $b$ once and $y$ is disjoint from $a, c, d$, we can choose $y' \in \lambda([y])$ such that $y'$ intersects $b'$ once and $y'$ is disjoint from $a', c', d'$. Since $a, y, c$ form an embedded triangle, $a', y', c'$ form an embedded triangle. Since $x$ intersects $c$ once and $x$ is disjoint from $a, b, d$, we can choose $x' \in \lambda([x])$ such that $x'$ intersects $c'$ once and $x'$ is disjoint from $a', b', d'$. Now if the identification of the sides of $c'$ is as shown in Figure \ref{10fig} (ii), then we would get a contradiction by Lemma \ref{embedded} as $a', x', b'$ do not form an embedded triangle while $a, x, b$ form an embedded triangle. Hence, the result of the lemma follows.
\end{proof}

\begin{lemma}
\label{11i}
Let $a, b, c, d$ be essential pairwise disjoint nonisotopic properly embedded arcs on $N$. Suppose that there exists a subsurface $K$ of $N$ and a homeomorphism $\phi: (K, a, b, c, d) \rightarrow (K_0, a_0, b_0, c_0, d_0)$ where $K_0$ and $a_0, b_0, c_0, d_0$ are as shown in \ref{11fig} (i). There exist $a' \in \lambda([a]), b' \in \lambda([b]), c' \in \lambda([c]), d' \in \lambda([d])$, $K' \subseteq N$ and a homeomorphism $\chi: (K', a', b', c', d') \rightarrow (K_0, a_0, b_0, c_0, d_0)$.
\end{lemma}

\begin{proof} Let $x_0, y_0$ be as shown in Figure \ref{11fig} (i). Let $x = \phi^{-1}(x_0), y = \phi^{-1}(y_0)$. Since $a, b, c, d$ are pairwise disjoint, $a, b, c$ form an embedded triangle, and $b, c, d$ form an embedded triangle, by using that $\lambda$ is injective and the results of Lemma \ref{embedded} we can choose
$a' \in \lambda([a]), b' \in \lambda([b]), c' \in \lambda([c]), d' \in \lambda([d])$ such that $a', b', c', d'$ are pairwise disjoint, $a', b', c'$ form an embedded triangle and $b', c', d'$ form an embedded triangle. Since $y$ intersects $b$ once and $y$ is disjoint from $a, c, d$, we can choose $y' \in \lambda([y])$ such that $y'$ intersects $b'$ once and $y'$ is disjoint from $a', c', d'$. Since $a, y, c$ form an embedded triangle, $a', y', c'$ form an embedded triangle. Since $x$ intersects $c$ once and $x$ is disjoint from $a, b, d$, we can choose $x' \in \lambda([x])$ such that $x'$ intersects $c'$ once and $x'$ is disjoint from $a', b', d'$. Now if the identification of the sides of $c'$ is as shown in Figure \ref{11fig} (ii), then we would get a contradiction by Lemma \ref{embedded} as $a', x', d'$ do not form an embedded triangle while $a, x, d$ form an embedded triangle. Hence, the result of the lemma follows.
\end{proof}

\begin{figure}
\begin{center}
\epsfxsize=2.7in \epsfbox{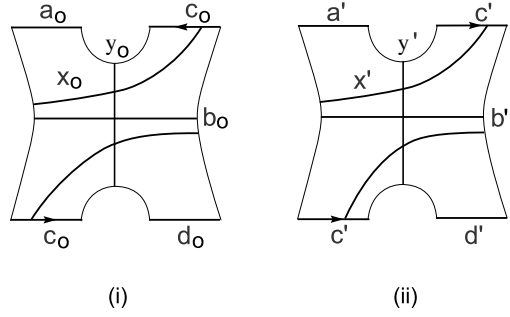}
\caption{Arc configurations VII}
\label{11fig}
\end{center}
\end{figure}

\begin{lemma}
\label{top} Suppose that $(g, r) \neq (1,1)$,  $(g, r) \neq (1,2)$ and $(g, r) \neq (2,1)$.
Let $\lambda :\mathcal{A}(N) \rightarrow \mathcal{A}(N)$ be an
injective simplicial map of $\mathcal{A}(N)$. $\lambda$ preserves the topological equivalence of ordered
triangulations on $N$ (i.e. for a given ordered triangulation $T = (a_1, a_2, \cdots, a_{3g+3r-6})$ on $N$, and a corresponding
ordered triangulation $T' = (a_1', a_2', \cdots, a_{3g+3r-6}')$ where $[a_i'] = \lambda([a_i])$ $\forall i= 1, 2,
\cdots, 3g+3r-6$, there exists a homeomorphism $h: N \rightarrow N$ such that $h(a_i) = a_i'$ $\forall i= 1, 2,
\cdots, 3g+3r-6$).\end{lemma}

\begin{proof} $N_T$ and $N_{T'}$ have both $2g + 2r - 4$ components. Let $\{\Delta_i: 1 \leq i \leq 2g+2b-4 \}$ be the distinct components of $N_T$. By Lemmas \ref{embedded}, \ref{regular}, \ref{twisted} we know that ``images'' of triangles of $T$ under $\lambda$ are triangles of $T'$, and the types of these triangles are preserved. Let $\Delta_i'$ be the component of $N_{T'}$ which corresponds to $\Delta_i$ under this correspondence.

Suppose that $\Delta_i$ is embedded. By Lemma \ref{embedded}, $\Delta'_i$ is embedded. Let $j_i$, $k_i$, and $l_i$ be the sides of $\Delta_i$ and $j'_i$, $k'_i$, and $l'_i$ be the corresponding sides of $\Delta'_i$. There exists a homeomorphism $g_i : (\Delta_i,j_i,k_i,l_i) \rightarrow (\Delta'_i,j'_i,k'_i,l'_i)$. The orientation type of $g_i$ (i.e. whether it is orientation-reversing or orientation-preserving) is fixed.

Suppose that $\Delta_i$ is regular non-embedded. By Lemma \ref{regular}, $\Delta'_i$ is regular non-embedded. Let $j_i$, $k_i$ be the sides of $\Delta_i$ such that $j_i$ joins two different boundary components of $\partial N$, and $k_i$ joins a component of $\partial N$ to itself. Let $j_i'$, $k_i'$ be the corresponding sides of $\Delta_i'$. By Lemma \ref{regular}, $j_i'$ joins two different boundary components of $\partial N$, and $k'_i$ joins a component of $\partial N$ to itself. There exist homeomorphisms $g_i, g_i^* : (\Delta_i,j_i,k_i) \rightarrow (\Delta'_i,j'_i,k'_i)$ such that $g_i$
is orientation preserving on $k_i$ and $g_i^*$ is orientation reversing on $k_i$.

Suppose that $\Delta_i$ is twisted non-embedded. By Lemma \ref{twisted}, $\Delta'_i$ is twisted non-embedded. Let $j_i$, $k_i$ be the sides of $\Delta_i$ such that $j_i$ joins two different boundary components of $\partial N$, and $k_i$ joins a component of $\partial N$ to itself. Let $j_i'$, $k_i'$ be the corresponding sides of $\Delta_i'$. By Lemma \ref{twisted}, $j_i'$ joins two different boundary components of $\partial N$, and $k'_i$ joins a component of $\partial N$ to itself. There exist homeomorphisms $g_i, g_i^* : (\Delta_i,j_i,k_i) \rightarrow (\Delta'_i,j'_i,k'_i)$ such that $g_i$
is orientation preserving on $k_i$ and $g_i^*$ is orientation reversing on $k_i$.

Lemmas \ref{9}, \ref{9i}, \ref{9ii}, \ref{10i}, \ref{11i} ensure that we can choose homeomorphisms $h_i : \Delta_i \rightarrow \Delta'_i$, $1 \leq i \leq 2g + 2r -4$, where $h_i$ is isotopic to $g_i$, if $\Delta_i$ is embedded, and $h_i$ is isotopic to either $g_i$ or $g^*_i$, if $\Delta_i$ is non-embedded, so that the unique homeomorphism $h : N_T \rightarrow N_{T'}$ whose restriction to $\Delta_i$ is equal to $h_i$, $1 \leq i \leq 2g + 2r -4$, covers a homeomorphism $h : N \rightarrow N$. By Lemmas \ref{9}, \ref{9i}, \ref{9ii}, \ref{10i}, \ref{11i}, the homeomorphisms $h_i : \Delta_i \rightarrow \Delta'_i$ and $h_j : \Delta_j \rightarrow \Delta'_j$ associated to embedded triangles $\Delta_i$ and $\Delta_j$ which have sides corresponding to the same element of $T$, can be isotoped by a relative isotopy to agree on that side. In other words, the restrictions of $h_i$ and $h_j$ to pairs of sides which correspond to the same element of $T$ have the same orientation type as such homeomorphisms between fixed elements of $T$ and $T'$.

When $\Delta_i$ is nonembedded, this condition on compatibility of orientation types of restrictions on pairs of sides which correspond to the same element of $T$ can be realized on all such pairs by making the appropriate choice of either $g_i$ or $g^*_i$.

Once the correct choices are made so that this compatibility of orientations is realized, we may isotope the chosen homeomorphisms, $h_i$, to homeomorphisms which agree, as homeomorphisms between fixed elements of $T$ and $T'$, on all pairs of sides which correspond to the same element of $T$. By gluing these homeomorphisms, we get a homeomorphism $h : N \rightarrow N$ which maps each element $j$ of $T$ to the corresponding element $j'$ of $T'$.
\end{proof}

\begin{theorem}
\label{last}
Suppose that $(g, r) \neq (1,1)$,  $(g, r) \neq (1,2)$ and $(g, r) \neq (2,1)$. If $\lambda : \mathcal{A}(N) \rightarrow \mathcal{A}(N)$ is an injective simplicial map, then it
is induced by a homeomorphism $h : N \rightarrow N$ (i.e $\lambda([a]) = [h(a)]$
for every essential properly embedded arc $a$ on $N$).\end{theorem}

\begin{proof} Let $\Delta$ be a top dimensional simplex in $\mathcal{A}(N)$. By Lemma \ref{top} there exists a homeomorphism $h : N \rightarrow N$
such that $\lambda$ agrees with the map induced by $h$ on every vertex in $\Delta$. We can see that $\lambda$ agrees with this map on every vertex of $\mathcal{A}(N)$
by following the proof of Lemma 4.13 in \cite{Ir1},
and using that (i) every vertex is contained in a top dimensional simplex, (ii) every codimension one face of a top dimensional simplex is contained in one or two top dimensional simplices, (iii) between any two top dimensional simplices $\Delta$ and $\Delta'$ of $\mathcal{A}(N)$ there exists a chain $\Delta = \Delta_0, \Delta_1, \cdots, \Delta_m = \Delta'$ of top dimensional simplices in $\mathcal{A}(N)$ connecting $\Delta$ to $\Delta'$ such that any two consecutive simplices $\Delta_i, \Delta_{i+1}$ have exactly one common face of codimension 1
(by Hatcher \cite{H}, (see also Mosher's proof for orientable case in \cite{Mos})).
\end{proof}

\begin{theorem} Suppose that $(g, r) \neq (1,1)$,  $(g, r) \neq (1,2)$ and $(g, r) \neq (2,1)$. Then
$Aut(\mathcal{A}(N)) \cong Mod_N $.
\end{theorem}

\begin{proof} Every element of $Mod_N$ induces an automorphism of $\mathcal{A}(N)$. Every automorphism of $\mathcal{A}(N)$ is induced by a homeomorphism of $N$ by Theorem \ref{last}. If $\lambda : \mathcal{A}(N) \rightarrow \mathcal{A}(N)$ fixes the isotopy class of every arc, then
it is induced by the identity homeomorphism. Hence, we see that the groups are isomorphic.
\end{proof}

\vspace{0.5cm}

{\bf Acknowledgments}

\vspace{0.3cm}

We thank Ursula Hamenstadt for her comments about the paper. We thank Peter Scott for some discussions
about this work and Lee Mosher for his interest in the paper.


\vspace{0.2cm}

\noindent Bowling Green State University, Department of Mathematics and Statistics, Bowling Green, OH, 43403, USA; eirmak@bgsu.edu.\\

\end{document}